\documentclass[11pt,a4paper]{article}
\usepackage[T1]{fontenc}
\usepackage{lmodern}
\usepackage{amsmath,amssymb,amsthm}
\usepackage[numbers,sort&compress]{natbib}
\usepackage[top=2.4cm,bottom=2.4cm,left=2.3cm,right=2.3cm]{geometry}
\usepackage{indentfirst,microtype}
\usepackage[colorlinks=true,linkcolor=blue,citecolor=magenta,urlcolor=blue]{hyperref}
\allowdisplaybreaks[1]
\newtheorem{theorem}{Theorem}[section]
\newtheorem{lemma}[theorem]{Lemma}
\newtheorem{conjecture}[theorem]{Conjecture}
\newcommand{\ex}{\operatorname{ex}}
\newcommand{\dunion}{\mathbin{\dot\cup}}
\newcommand{\cost}{\partial_r}

\title{Exact Clique Counts in Graphs with a Forbidden Tree}
\author{Xiaojun Zhao\thanks{School of Mathematics, Hunan University, Changsha 410082, P.R. China. Email: xiaojunzhao@hnu.edu.cn.}
\and Yuejian Peng\thanks{Corresponding author. School of Mathematics, Hunan University, Changsha 410082, P.R. China. Email: ypeng1@hnu.edu.cn.}}
\hypersetup{pdftitle={Large Clique Counts in Graphs with a Forbidden Tree},pdfauthor={Xiaojun Zhao and Yuejian Peng}}

\begin{document}
\setlength{\baselineskip}{14pt}
\maketitle

\begin{abstract}
For a fixed tree $T$ on $t$ vertices, let $\ex(n,K_r,T)$ denote the maximum number of $r$-cliques in an $n$-vertex $T$-free graph.   Write $n=\alpha(t-1)+\beta$ with $0\le\beta<t-1$. A conjecture proposed in extremal-graph form by Alon and Shikhelman~\cite{alon2016}, later stated explicitly as an exact clique-counting formula by Gerbner and Palmer~\cite{GP2026},  asserts that
\[
\ex(n,K_r,T)=\alpha\binom{t-1}{r}+\binom\beta r,
\]
with the extremal construction given by the vertex disjoint union of  $\alpha$  copies of $K_{t-1}$ and a graph with $\beta$ vertices.
This paper proves this conjecture for every $n$  and for every tree $T$ whenever $C\log t\le r\le t-1$ and $t$ is sufficiently large, where $C$ is an absolute constant. This lowers the previously known general threshold $r\ge\lceil(2t-1)/3\rceil$, which is linear in $t$ to logarithmic order. The paper  also proves the conjecture for every $3\le r\le t-1$ when at least $\lceil(t-3)/2\rceil$ leaves of $T$ have a common neighbor, and this extends the known results for stars and double stars.

\end{abstract}

\section{Introduction}\label{intro}

All graphs are finite, simple and undirected. For graphs $H$ and $F$, the generalized Tur\'{a}n number $\ex(n,H,F)$ is the largest number of copies of $H$ in an $n$-vertex graph containing no copy of $F$. Copies need not be induced, and $H=K_2$ gives the usual Tur\'{a}n problem. Clique counting goes back to Zykov~\cite{zykov1949}; Alon and Shikhelman~\cite{alon2016} began a systematic study of generalized Tur\'{a}n problems. We refer to Gerbner and Palmer~\cite{GP2026} for a survey.

 The Erd\H{o}s--S\'{o}s conjecture~\cite{erdos1962} states that every graph of average degree greater than $t-2$ contains every $t$-vertex tree.   Several partial results were obtained~\cite{1959Erdos, ajtai1994, Brandt1996,goerlich2016,rozhon2019,besomi2019,Fan2013,Fan2007,Fan2019,sacle1997} and the  general conjecture had remained open for several decades  until
 the AI system GPT-6 Astra recently  found a proof,  Riordan and Scott~\cite{RiordanScott2026} simplified the argument and determined the edge extremal graphs, while Wood~\cite{Wood2026} gave an exposition. In this paper, we study the  clique  counterpart of this problem.

For integers $M\ge r\ge3$ and $n\ge0$, write $n=\alpha M+\beta$, where $0\le\beta<M$, and let
\begin{equation}\label{F-definition}
F_{M,r}(n)=\alpha\binom Mr+\binom\beta r.
\end{equation}
We use the convention $\binom{x}{j}=0$ for nonnegative integers $x$ when $j<0$ or $j>x$. The graph $\alpha K_M\dunion K_\beta$, where $\dunion$ denotes  the vertex disjoint union and $\alpha K_M$ is the vertex disjoint union of $\alpha$ copies of $K_M$, contains no tree on $M+1$ vertices and has exactly $F_{M,r}(n)$ copies of $K_r$. Alon and Shikhelman~\cite[Section~6]{alon2016}  proposed that, when a tree on $t$ vertices is forbidden, an extremal graph maximizing the number of cliques is a vertex-disjoint union of complete graphs, all but possibly one of order $t-1$. This proposed extremal construction yields exactly the formula $F_{t-1,r}(n)$ above. Gerbner and Palmer~\cite[Conjecture~6.5]{GP2026} subsequently stated the exact identity as a formal conjecture.

\begin{conjecture}[Alon--Shikhelman~\cite{alon2016}; see also Gerbner--Palmer~\cite{GP2026}]\label{C1}
Let $T$ be a tree on $t$ vertices and let $3\le r\le t-2$. Then, for every $n\ge0$,
\[
\ex(n,K_r,T)=F_{t-1,r}(n).
\]
\end{conjecture}

The conjecture is confirmed for several basic families of trees. For the path $P_t$ on $t$ vertices, Luo~\cite{luo2018maximum} proved the linear bound
\begin{equation}\label{linear-bound}
\ex(n,K_r,P_t)\le\frac{n}{t-1}\binom{t-1}{r},
\end{equation}
and Chakraborti and Chen~\cite{chakraborti2024exact} determined the exact value for every $n$, together with all extremal graphs. For stars, the problem is equivalent to counting cliques under a maximum-degree bound. The exact formula was conjectured by Gan, Loh and Sudakov~\cite{gan2015maximizing}, proved by Chase~\cite{chase2020}, and given a unified proof by Chao and Dong~\cite{Chao2022}. Gerbner~\cite{gerbner2023doublestars} confirmed the conjecture for double stars. For arbitrary trees, Zhou and Yuan~\cite{ZhouYuan2026} proved it for $r=t-2$ and $r=t-3\ge5$, including all equality cases.

In our earlier version~\cite[Theorem~1.6]{ZhaoPeng2026v1} (which was not submitted for journal publication), we applied a clique-partition method developed from our earlier work \cite{ZP1} and proved the exact formula and all equality cases for $r=t-d$ under $d\ge2$ and $t\ge d^2-d+3$. This allows the deficit $d=t-r$ to grow on the square-root scale in $t$. Cheng and Tang~\cite[Theorem~1.2]{ChengTang2026} extended the result to $r\ge\lceil(2t-1)/3\rceil$, or equivalently $d\le\lfloor(t+1)/3\rfloor$, so the deficit can grow linearly with $t$. The earlier version also treated clique spectral radius; those results are developed separately in~\cite{ZhaoPengSpectral2026}.

Gerbner, Methuku and Palmer~\cite{GMP2020} showed that the Erd\H{o}s--S\'{o}s edge bound for $T$ and its subtrees implies~\eqref{linear-bound} with $P_t$ replaced by $T$, for $3\le r\le t-3$. The recent theorem~\cite{RiordanScott2026,Wood2026} therefore gives this linear clique bound for every tree. It determines the exact value when $t-1$ divides $n$, but for $M=t-1$ and $n=\alpha M+\beta$ the gap is
\[
\frac{n}{M}\binom Mr-F_{M,r}(n)
=\frac{\beta}{M}\binom Mr-\binom\beta r,
\]
which is positive when $0<\beta<M$. Closing this gap and proving that complete components of order $t-1$ are forced require more than the linear estimate.

Our first result proves the conjecture for every tree once the clique order is at least a constant times the logarithm of the tree order. All logarithms are natural.

\begin{theorem}\label{T5}
There are absolute constants $C>0$ and $t_0$ such that, for every tree $T$ on $t\ge t_0$ vertices and every integer $r$ with
$C\log t\le r\le t-1$,
we have $\ex(n,K_r,T)=F_{t-1,r}(n)$ for every $n\ge0$. Let $n=\alpha(t-1)+\beta$ with $0\le\beta<t-1$, then the extremal graphs are exactly
$\alpha K_{t-1}\dunion H$,
where $H=K_\beta$ when $\beta\ge r$, and $H$ is any graph on $\beta$ vertices when $\beta<r$.
\end{theorem}

For sufficiently large $t$, Theorem~\ref{T5} lowers the required clique order from the linear threshold  in \cite{ChengTang2026} to a logarithmic one.  The proof permits $C=8/\gamma$, where $0<\gamma\le1/4$ is an absolute constant from the embedding theorem of Havet, Reed, Stein and Wood~\cite{HRSW2020}.

Our second result covers every clique order for trees with many leaves at one vertex. Define
\[
\lambda(T)=\max_{v\in V(T)}|\{u\in N_T(v):d_T(u)=1\}|,
\]
i.e., $\lambda(T)$ is the largest number of leaves with a common neighbor, where $N_T(v)$ denotes the neighborhood of $v$ in $T$.

\begin{theorem}\label{leaf-bunch-counting-theorem}
Let $T$ be a tree on $t\ge4$ vertices such that
$\lambda(T)\ge\left\lceil\frac{t-3}{2}\right\rceil$.
For every $3\le r\le t-1$ and every $n\ge0$, we have $\ex(n,K_r,T)=F_{t-1,r}(n)$.
The extremal graphs are exactly
$\alpha K_{t-1}\dunion H$,
where $H=K_\beta$ when $\beta\ge r$, and $H$ is any graph on $\beta$ vertices when $\beta<r$.
\end{theorem}

The condition holds for every star and double star. This theorem extends the known star and double-star cases and complements Theorem~\ref{T5}.

 Theorem~\ref{T5} and  Theorem~\ref{leaf-bunch-counting-theorem} are proved in  Section~\ref{log-section} and Section~\ref{leaf-section} respectively. The main idea in the proof is to compare the number of cliques lost by deleting vertices with the change in the proposed extremal bound. The strict gap in this comparison also determines all extremal graphs. Combined with a known tree-embedding theorem, the method yields the logarithmic threshold. For trees with many leaves sharing a common parent, we use a more careful deletion method to cover all clique orders $r\ge 3$.

\section{Proof of Theorem~\ref{T5}}
\label{log-section}

 We write $|G|=|V(G)|$, $d_G(v)=|N_G(v)|$, and $N_G[v]=N_G(v)\cup\{v\}$. For $X\subseteq V(G)$, let $G-X$ be the subgraph induced by $V(G)\setminus X$. Let $c_r(G)$ be the number of the $r$-cliques of $G$, and let $c_r(v)$ be the number of $r$-cliques containing $v$. When $M$ and $r$ are fixed, write $F(n)=F_{M,r}(n)$. Also let
\[
\cost(X;G)=c_r(G)-c_r(G-X)
\]
 be the number of $r$-cliques meeting $X$.

\subsection{ Preliminary lemmas}\label{basic-section}

In this section, we give the counting tools and preliminary lemmas.

\begin{lemma}\label{isolation}
Let \(T\) be a tree on \(t\ge2\) vertices. If a \(T\)-free graph \(G\) contains a copy \(Q\) of \(K_{t-1}\), then \(Q\) is a component of \(G\).
\end{lemma}

\begin{proof}
Let \(\ell\) be a leaf of \(T\) and let \(p\) be its neighbor. If \(xy\in E(G)\) with \(x\in V(Q)\) and \(y\notin V(Q)\), embed \(T-\ell\) into \(Q\) with \(p\) mapped to \(x\), and map \(\ell\) to \(y\). This gives a copy of \(T\), a contradiction.
\end{proof}

\begin{lemma}\label{opt}
Let $M\ge r\ge3$. If $n_1,\ldots,n_l$ are positive integers with $n_i\le M$ and $\sum_i^l n_i=n$, then
\[
\sum_{i=1}^l\binom{n_i}{r}\le F_{M,r}(n),
\]
After parts smaller than $r$ are replaced by singleton parts,
and the equality is possible only when there are $\lfloor n/M\rfloor$ parts of size $M$, together with one part of size $n\bmod M$ if that remainder is at least $r$, and singleton parts otherwise. The function $F_{M,r}$ is nondecreasing, and
\begin{equation}\label{F-superadditive}
F_{M,r}(x)+F_{M,r}(y)\le F_{M,r}(x+y)
\qquad(x,y\ge0).
\end{equation}
\end{lemma}

\begin{proof}
If two parts have sizes $p\ge q\ge r$ and $p<M$, transferring vertices from the smaller part to the larger one increases the sum $\sum_{i=1}^l\binom{n_i}{r}$ until one part has size $M$ or the smaller part has size less than $r$. The increase is strict, because at each relevant step the increase is
\[
(\binom {p+1}{r}+\binom {q-1}{r})-(\binom {p}{r}+\binom {q}{r})=\binom p{r-1}-\binom{q-1}{r-1}>0.
\]
A part smaller than $r$ contributes zero to the sum  and can be replaced by singletons. If a part has size at least $r$ but less than $M$, adding an available singleton strictly increases its contribution to the sum. After full parts of size $M$ are set aside, any $r$ remaining singletons can be combined into a new part, again with a strict increase. Repeating these operations gives the stated form and the equality cases. The case $M=r$ also follows by counting parts of size $M$ directly.

Monotonicity follows directly from the definition of  $F_{M, r}(n)$ (see \eqref{F-definition}). To prove~\eqref{F-superadditive}, write both arguments as multiples of $M$ plus their remainders and apply the same comparison to the two remainder parts. This also proves
\begin{equation}\label{bounded-degree-convex-step}
F(n-q)+\binom qr\le F(n)\qquad(0\le q\le \min\{M,n\}).\qedhere
\end{equation}
\end{proof}

The next lemma keeps a deficit term in the local counting argument of Chao and Dong~\cite{Chao2022}. This term is needed in the proof of Lemma~\ref{degree-transfer}.

\begin{lemma}\label{local-sum-strong}
Let $r\ge3$ be an integer and let $G$ be a graph. For $v\in V(G)$, let $\tau_r(v)$ be the number of $r$-cliques meeting $N_G[v]$, and let
\[
a_r(v)=\binom{d_G(v)}{r-1}-c_r(v).
\]
Then
\begin{equation}\label{local-sum-defect}
\sum_v\tau_r(v)\le
\sum_v\left(\binom{d_G(v)+1}{r}-a_r(v)\right).
\end{equation}
If $G$ is nonempty, then in particular some vertex $v$ satisfies the corresponding pointwise inequality.
\end{lemma}

\begin{proof}
After subtracting $\sum_v c_r(v)$ from both sides of (\ref{local-sum-defect}), the new left side counts pairs consisting of an $r$-clique and an adjacent vertex outside that clique. The new right side becomes $\sum_v\binom{d_G(v)}r$, which counts stars with $r$ distinct leaves; edges among the leaves are allowed.

Let us compare the two counts on each set of $r+1$ vertices. If that set contains no $r$-clique, its  contribution to the new left side is zero. Otherwise, fix an $r$-clique and let the remaining vertex have $q$ neighbors in it. When $q=0$,  it contributes to both sides by zero. When $1\le q\le r-2$, the contributions to the left side and the right side are  $1$ and $q$ respectively. When $q=r-1$, the contributions to the left side and the right side are $2$ and $r-1$ respectively. When $q=r$, the contributions are both $r+1$. In every case the contribution to the left side is at most the contribution to the right side, since $r\ge3$.

This proves~\eqref{local-sum-defect}.
\end{proof}

The following lemma characterizes the extremal graphs in  the result by Chao and Dong~\cite{Chao2022} and Chase~\cite{chase2020}.

\begin{lemma}\label{bounded-degree-clique-equality}
Let $M\ge r\ge3$, and let $H$ be a graph on $m$ vertices with $\Delta(H)\le M-1$. Then
\[
c_r(H)\le F_{M,r}(m).
\]
If every edge of $H$ belongs to an $r$-clique, equality holds exactly when the nontrivial components are $a$ copies of $K_M$ and, if $b\ge r$, one copy of $K_b$, where $m=aM+b$ and $0\le b<M$. All other vertices are isolated. In particular,
\begin{equation}\label{bounded-strict}
m\ge M,\quad K_M\nsubseteq H
\quad\Longrightarrow\quad c_r(H)<F_{M,r}(m).
\end{equation}
\end{lemma}

\begin{proof}
The upper bound is the theorem of Chao and Dong~\cite{Chao2022}. It also follows directly by induction from Lemma~\ref{local-sum-strong}: select $v$ with $\tau_r(v)\le\binom{d_H(v)+1}r$, delete its closed neighborhood, and use~\eqref{bounded-degree-convex-step}.

Suppose equality holds. Fix any vertex $v$ and put $s=d_H(v)+1\le M$. The upper bound just proved and~\eqref{bounded-degree-convex-step} give
\[
F(m)=c_r(H)=c_r(H-N_H[v])+\tau_r(v)
\le F(m-s)+\tau_r(v)
\le F(m)-\binom sr+\tau_r(v).
\]
Thus $\tau_r(v)\ge\binom{d_H(v)+1}r$ for every $v$. Summing and applying~\eqref{local-sum-defect}, with $a_r(v)\ge0$, gives $\sum_v a_r(v)=0$. Hence
\[
c_r(v)=\binom{d_H(v)}{r-1}\qquad(v\in V(H)).
\]
Because every edge lies in an $r$-clique, each nonisolated vertex has degree at least $r-1$. The identity $c_r(v)=\binom{d_H(v)}{r-1}$ then forces every $(r-1)$-subset of its neighborhood to be a clique. Since $r\ge3$, its neighborhood is complete. Consequently, each nontrivial component is complete, with order between $r$ and $M$. By Lemma~\ref{opt}, the conclusion holds.

For a general graph $H$, delete edges that belong to no $r$-clique. This preserves $c_r(H)$ and cannot create $K_M$ or increase the maximum degree. The equality statement then gives~\eqref{bounded-strict}.
\end{proof}

The following lemma gives a finite condition that converts a degree bound into the exact clique formula.

\begin{lemma}\label{degree-transfer}
Let $M$ and $D$ be integers with $M\ge3$ and $0\le D\le M-2$, and let $T$ be a tree of order $M+1$. Suppose that every nonempty $T$-free graph $J$ satisfies
\begin{equation}\label{degree-assumption}
\delta(J)\ge D+1\quad\Longrightarrow\quad\Delta(J)\le M-1.
\end{equation}
If $r$ is an integer with $3\le r\le M$ and
\begin{equation}\label{degree-transfer-condition}
M\binom D{r-1}<\binom{M-3}{r-3},
\end{equation}
then $\ex(n,K_r,T)\leq F_{M,r}(n)$ for every $n$. Its equality cases are those in Theorem~\ref{T5}.
\end{lemma}

\begin{proof}[Proof of Lemma~\ref{degree-transfer}]
Let $G$ be a $T$-free graph on $n$ vertices. We use strong induction on $n$, and also prove that the upper bound is strict when $n\ge M$ and $G$ has no $K_M$. The case $n<M$ follows from the complete-graph bound. For $n=M$, the same bound gives $c_r(G)\le\binom Mr=F(M)$, with strict inequality if $G$ has no $K_M$, since then some $r$-subset is not a clique. If $G$ has a $K_M$, it is a component by Lemma~\ref{isolation}, so we can remove it and apply induction. Hence assume that $n\ge M$ and $G$ has no $K_M$.

If $\binom D{r-1}=0$, repeatedly remove a vertex of current degree at most $D$. No $r$-clique is lost. If the graph left is empty, strictness is immediate. Otherwise it has maximum degree at most $M-1$, by~\eqref{degree-assumption}. If it has $m\ge M$ vertices, then~\eqref{bounded-strict} and monotonicity give $c_r(G)<F(m)\le F(n)$; if $m<M$, then $c_r(G)\le\binom{M-1}r<\binom Mr\le F(n)$.

Suppose now that $\binom D{r-1}>0$. Remove vertices of current degree at most $D$ until $M$ vertices have been removed or no such vertex remains. Each removal loses at most $\binom D{r-1}$ cliques. If a set $P$ of $M$ vertices is removed, then induction gives
\[
c_r(G)\le c_r(G-P)+M\binom D{r-1}
\le F(n-M)+M\binom D{r-1}
<F(n-M)+\binom Mr=F(n),
\]
where we used $M\binom D{r-1}<\binom{M-3}{r-3}\le\binom Mr$ and $F(n)=F(n-M)+\binom Mr$.

Otherwise, let $P$ be the set removed, put $q=|P|<M$, and let $H=G-P$. The graph $H$ is nonempty and satisfies
\[
D+1\le\delta(H)\le\Delta(H)\le M-1.
\]
If $q=0$, use~\eqref{bounded-strict}. Assume $1\le q<M$. By Lemma~\ref{local-sum-strong}, there is a vertex $v$ of $H$ such that, with $s=d_H(v)+1$,
\[
\tau_r(v)\le\binom sr-a_r(v),\qquad D+2\le s\le M.
\]
If $s=M$, the neighborhood of $v$ is not complete, since $H$ has no $K_M$. A missing edge belongs to $\binom{M-3}{r-3}$ subsets of order $r-1$, so
\begin{equation}\label{full-neighborhood-loss}
a_r(v)\ge \binom{M-3}{r-3}.
\end{equation}
Delete $W=P\cup N_H[v]$. Counting first the cliques meeting $P$ and then the remaining cliques meeting $N_H[v]$ gives
\[
\cost(W;G)\le q\binom D{r-1}+\binom sr-a_r(v).
\]
If $s+q\le M$, recall that $D\ge r-1$, we have
\[
F(s+q)-\binom sr
=\sum_{i=0}^{q-1}\binom{s+i}{r-1}
\ge q\binom s{r-1}>q\binom D{r-1}.
\]
If $s+q>M$ and $s<M$, then
\[
F(s+q)-\binom sr\ge\binom Mr-\binom sr
\ge\binom{M-1}{r-1}\ge \binom{M-3}{r-3}>M\binom D{r-1}>q\binom D{r-1}.
\]
Finally, if $s=M$,~\eqref{full-neighborhood-loss} gives
\[
q\binom D{r-1}+\binom Mr-a_r(v)<\binom Mr\le F(M+q).
\]
Thus in every case $\cost(W;G)<F(|W|)$. Applying induction to $G-W$ and using~\eqref{F-superadditive} proves $c_r(G)<F(n)$.

It follows that an extremal graph of order at least $M$ contains a $K_M$, which is a component by Lemma~\ref{isolation}. Removing it preserves equality, since $F(n)-F(n-M)=\binom Mr$. Repeating this step yields exactly $\alpha$ components $K_M$ and a remainder $H$ on $\beta$ vertices. If $\beta\ge r$, equality forces $c_r(H)=\binom\beta r$, hence $H=K_\beta$; if $\beta<r$, every graph $H$ contributes no $r$-clique. Conversely, all the stated graphs attain the bound.
\end{proof}

\subsection{Proof of Theorem~\ref{T5}}

Havet, Reed, Stein and Wood~\cite[Theorem~1.3]{HRSW2020} proved that there is an absolute constant $\gamma>0$ such that every graph with minimum degree at least $(1-\gamma)m$ and maximum degree at least $m$ contains every tree with $m$ edges. We apply this result with $m=t-1$.

\begin{proof}[Proof of Theorem~\ref{T5}]
Decrease $\gamma$ if necessary so that $0<\gamma\le1/4$. Put $M=t-1$ and
\[
D=\lceil(1-\gamma)M\rceil-1.
\]
For sufficiently large $M$, we have $0\le D\le M-2$. The contrapositive of this embedding theorem gives~\eqref{degree-assumption} for every $T$-free graph.

If $r-1>D$, condition~\eqref{degree-transfer-condition} is automatic. Otherwise, for $M\ge2/\gamma$,
\[
\frac{D}{M-1}\le1-\frac\gamma2,
\]
and a termwise comparison of binomial coefficients gives
\begin{align*}
\frac{M\binom D{r-1}}{\binom{M-3}{r-3}}
&=\frac{M(M-1)(M-2)}{(r-1)(r-2)}
\frac{\binom D{r-1}}{\binom{M-1}{r-1}}\\
&\le\frac{M^3}{2}\left(1-\frac\gamma2\right)^{r-1}
\le\frac{M^3}{2}\exp\left(-\frac\gamma2(r-1)\right).
\end{align*}
The last expression is at most $1/2$ whenever
\[
r\ge\left\lceil1+\frac6\gamma\log M\right\rceil.
\]
Take $C=8/\gamma$. For $t\ge3$, we have $C\log t\ge1+(6/\gamma)\log(t-1)$. Choose $t_0$ so that $t-1\ge2/\gamma$ and $C\log t\le t-1$ for every $t\ge t_0$. Every integer $r\ge C\log t$ then satisfies the required inequality. Lemma~\ref{degree-transfer} gives the formula and all equality cases.
\end{proof}

\section{Proof of Theorem~\ref{leaf-bunch-counting-theorem}}\label{leaf-section}

We first show that a vertex of large degree in a $T$-free graph forces a sequence of vertices that can be removed at small cost. The forest left after deleting a vertex of $T$ and its leaf neighbors may have any shape.

\begin{lemma}\label{leaf-bunch-delete-many}
Let $T$ be a tree of order $M+1$, and let $\mathcal F$ be obtained by deleting a vertex $p$ and all its leaf neighbors. Suppose that $\mathcal F$ has $h\ge2$ vertices and $k$ components; thus $1\le k\le\lfloor h/2\rfloor$. If $G$ is $T$-free and $d_G(u)\ge M$, then there are distinct vertices $v_1,\ldots,v_q\in V(G)\setminus\{u\}$, where $q=M-k+1$, such that
\[
d_{G-\{v_1,\ldots,v_{i-1}\}}(v_i)\le h-1
\qquad(1\le i\le q).
\]
Consequently, for $P=\{v_1,\ldots,v_q\}$ and every $r\ge3$,
\[
\cost(P;G)\le q\binom{h-1}{r-1}.
\]
\end{lemma}

\begin{proof}
Note that each component of $\mathcal F$ has at least two vertices since an isolated component would be a leaf neighbor of $p$ and hence would have been deleted. Also, note that each component has a unique root, namely the vertex formerly adjacent to $p$. Keep $u$ and repeatedly delete another vertex of current degree at most $h-1$.  It is sufficient to show that this process can continue at least $q$ times. Suppose on the contrary that  this process stops after $z<q$ deletions. Then the remaining graph $J$ satisfies
\[
d_J(u)\ge M-z\ge k,
\qquad \delta(J-u)\ge h-1.
\]
Since $J-u$ is nonempty and $\delta(J-u)\ge h-1$, it has at least $h$ vertices, so there are enough vertices to embed $\mathcal F$. Map the $k$ roots to distinct vertices of $N_J(u)$ and regard all root images as reserved. Embed each rooted component of $\mathcal F$ greedily away from its root in $J-u$. When a non-root vertex is to be embedded, its parent has at least $h-1$ neighbors in $J-u$, while at most $h-2$ of them are already used or reserved for the other vertices of $\mathcal F$. Hence an unused neighbor is available, and the prescribed root images are respected throughout.

If $R$ is the image of $\mathcal F$, then $|N_G(u)\setminus R|\ge M-h$. These vertices supply the $M-h$ leaf neighbors of $p$, completing a copy of $T$ with $p$ mapped to $u$. Here the leaves are chosen in the original graph $G$, so previously deleted vertices remain available. This contradiction proves the deletion assertion, and summing the clique losses proves the bound.
\end{proof}

The next inequality compares the deletion cost with the exact change in $F_{M,r}$.

\begin{lemma}\label{leaf-bunch-budget}
Let $h\ge2$, $1\le k\le\lfloor h/2\rfloor$ and $M\ge\max\{3,2h-2\}$ be integers. For $3\le r\le M$, put
\[
\Delta_r(M,h,k)=\binom Mr+\binom{h-1}{r}
-\binom{h+k-2}{r}-(M-k+1)\binom{h-1}{r-1}.
\]
Then $\Delta_r(M,h,k)\ge0$, with equality if and only if
\[
(M,h,k,r)=(4,3,1,3).
\]
\end{lemma}

\begin{proof}
For $h=2$, the expression is $\binom Mr>0$. Assume $h\ge3$. The identities
\begin{align*}
\Delta_3(M+1,h,k)-\Delta_3(M,h,k)
&=\binom M2-\binom{h-1}{2}>0,\\
\Delta_3(M,h,k+1)-\Delta_3(M,h,k)
&=\binom{h-1}{2}-\binom{h+k-2}{2}\le0
\end{align*}
reduce the triangle case to $M=2h-2$ and $k=\lfloor h/2\rfloor$. Direct expansion gives
\[
6\Delta_3(2h-2,h,\lfloor h/2\rfloor)=
\begin{cases}
3a(a-1)(3a-4),&h=2a,\ a\ge2,\\
3(a-1)(3a^2+3a-2),&h=2a+1,\ a\ge1.
\end{cases}
\]
These values are nonnegative, and equality occurs only when $(M,h,k)=(4,3,1)$.

Since $M>h+k-2$, the case $r\ge h+1$ gives
$\Delta_r=\binom Mr-\binom{h+k-2}{r}>0$.
For $4\le r\le h$, put $a=h-1$ and, for $s\ge a$, define
$R_j(s)=\binom{s}{j}/\binom{a}{j}$, $0\le j\le a$.
Its increments
\[
R_{j+1}(s)-R_j(s)=R_j(s)\frac{s-a}{a-j}
\qquad(0\le j<a)
\]
are nonnegative and nondecreasing, since both factors on the right are nonnegative and nondecreasing in $j$. As $R_0(s)=1$, the averages of these increments give
\[
\frac{R_j(s)-1}{j}
=\frac1j\sum_{i=0}^{j-1}\bigl(R_{i+1}(s)-R_i(s)\bigr)
\ge\frac{R_2(s)-1}{2}
\qquad(2\le j\le a).
\]
Using
\[
\binom Mr-\binom{h+k-2}{r}
=\sum_{s=h+k-2}^{M-1}\binom{s}{r-1}
\]
and $\binom ar/\binom a{r-1}=(h-r)/r$, we obtain
\begin{align*}
\frac{\Delta_r(M,h,k)}{\binom a{r-1}}
&=\sum_{s=h+k-2}^{M-1}R_{r-1}(s)
+\frac{h-r}{r}-(M-k+1)\\
&=\sum_{s=h+k-2}^{M-1}\bigl(R_{r-1}(s)-1\bigr)
-\frac{(r-1)h}{r},
\end{align*}
since the sum has $M-h-k+2$ terms.
For $r=3$, the same identity gives
\[
\sum_{s=h+k-2}^{M-1}\bigl(R_2(s)-1\bigr)
=\frac{\Delta_3(M,h,k)}{\binom a2}+\frac{2h}{3}.
\]
Applying the preceding bound with $j=r-1$ therefore yields
\begin{align}
\frac{\Delta_r(M,h,k)}{\binom a{r-1}}
&\ge\frac{r-1}{2}\frac{\Delta_3(M,h,k)}{\binom a2}
+(r-1)h\left(\frac13-\frac1r\right)>0.
\label{leaf-bunch-budget-lift}
\end{align}
Together with the triangle case, this proves the statement.
\end{proof}

\begin{proof}[Proof of Theorem~\ref{leaf-bunch-counting-theorem}]
Put $M=t-1$ and fix $3\le r\le M$. If $T$ is a star, every $T$-free graph has maximum degree at most $M-1$; Lemmas~\ref{bounded-degree-clique-equality} and~\ref{isolation} give the bound and its equality cases. The case $T=P_5$, including its equality cases, follows from the exact path theorem of Chakraborti and Chen~\cite[Theorem~1.6]{chakraborti2024exact}.

Otherwise, choose $p$ with $\ell=\lambda(T)$ leaf neighbors. Deleting $p$ together with these leaves gives a forest on $h=M-\ell\ge2$ vertices with $k\le\lfloor h/2\rfloor$ components. The hypothesis gives $M\ge2h-2$. The exceptional parameters $(M,h,k)=(4,3,1)$ would force $T=P_5$: every other tree on five vertices has two leaves with a common neighbor. Thus Lemma~\ref{leaf-bunch-budget} gives $\Delta_r(M,h,k)>0$.

Write $F(n)=F_{M,r}(n)$. We prove the upper bound for every $T$-free graph $G$ by strong induction on $n$, together with
\begin{equation}\label{leaf-bunch-strict}
n\ge M,\quad K_M\nsubseteq G
\quad\Longrightarrow\quad c_r(G)<F(n).
\end{equation}
For $n<M$, the claim follows from the complete-graph bound. For $n=M$, we have $c_r(G)\le\binom Mr=F(M)$, with strict inequality when $K_M\nsubseteq G$, since then some $r$-subset is not a clique. A copy of $K_M$ is a component by Lemma~\ref{isolation}, and its removal reduces both sides by $\binom Mr$. We may therefore assume $n>M$ and $K_M\nsubseteq G$. If $\Delta(G)\le M-1$, use~\eqref{bounded-strict}.

Otherwise choose $u$ with $d_G(u)\ge M$ and take the deletion sequence in Lemma~\ref{leaf-bunch-delete-many}. Put $q=M-k+1$ and write $n=\alpha M+\beta$, $0\le\beta<M$.
If $\beta\ge h$, then
\[
\cost(\{v_1\};G)\le \binom{h-1}{r-1}\le\binom{\beta-1}{r-1}=F(n)-F(n-1).
\]
Since $G-v_1$ has no $K_M$ and $n-1\ge M$, strict induction completes this case.

If $\beta<h$, then $q-\beta\geq M-k-h+2\geq 2h-2-k-h+2= h-k>0$ and $0\le\beta+k-1<M$. The function
$x\mapsto\binom{x+k-1}{r}-\binom xr$ is nondecreasing on nonnegative integers, so
\begin{align*}
F(n)-F(n-q)
&=\binom Mr+\binom\beta r-\binom{\beta+k-1}{r}\\
&\ge\binom Mr+\binom{h-1}{r}-\binom{h+k-2}{r}\\
&=q\binom{h-1}{r-1}+\Delta_r(M,h,k)>q\binom{h-1}{r-1}\ge\cost(P;G),
\end{align*}
where $P=\{v_1,\ldots,v_q\}$. Applying induction to $G-P$ proves~\eqref{leaf-bunch-strict}.

The graph $\alpha K_M\dunion K_\beta$ attains the bound. Conversely, if an extremal graph has at least $M$ vertices,~\eqref{leaf-bunch-strict} forces a copy of $K_M$, which is a component by Lemma~\ref{isolation}. Removing this component preserves equality because $F(n)-F(n-M)=\binom Mr$. Iterating yields exactly $\alpha$ components $K_M$ and a remainder $H$ on $\beta$ vertices. If $\beta\ge r$, equality requires $c_r(H)=\binom\beta r$, so $H=K_\beta$; if $\beta<r$, the remainder is unrestricted.
\end{proof}

\section{Concluding remarks}\label{concluding-remarks}

The main remaining question is Conjecture~\ref{C1} for arbitrary trees and small clique orders, especially triangles. At $r=3$, condition~\eqref{degree-transfer-condition} becomes $M\binom D2<1$, forcing $D\le1$, which the general embedding theorem does not provide. Further progress requires a smaller deletion cost or a counting argument that uses more of the structure of $T$. The proof of Theorem~\ref{leaf-bunch-counting-theorem} shows how a common neighbor of many leaves can supply such structure.

\section*{Acknowledgments}
We thank Daniel Gerbner for valuable communications. This research is supported in part by the National Natural Science Foundation of China (Nos.~12571363 and 12371327) and the Natural Science Foundation of Hunan Province (Grant No.~2025JJ30003).

\section*{Declaration on the use of generative AI}
We started this project in January 2025 and did not use any AI tool in the proof of the first version of this paper \cite{ZhaoPeng2026v1}. In the first version (which was not submitted for journal publication), we applied the clique partition method rooted from our earlier work \cite{ZP1} and this method is designed for counting large cliques. In the current version, we used
ChatGPT5.6 Sol to improve the earlier result. ChatGPT5.6 Sol was also used to improve the language. We take full responsibility for the content of the paper.

\setlength{\bibsep}{4pt}

\end{document}